\documentclass[11pt]{article}
\usepackage{graphicx} 

\usepackage[utf8]{inputenc}

\usepackage{amsmath}
\usepackage{mathtools}
\usepackage{amssymb}
\usepackage{amsthm}
\usepackage{graphicx}
\usepackage{amscd}
\usepackage{epic, eepic}
\usepackage[lowtilde]{url}
\usepackage{color}
\usepackage[utf8]{inputenc} 
\usepackage{makebox}
\usepackage{comment}
\usepackage{enumerate}   
\usepackage{hyperref}
\usepackage{todonotes}

\theoremstyle{plain}
\newtheorem{theorem}{\bf Theorem}[section]
\newtheorem{lemma}[theorem]{\bf Lemma}

\newtheorem{prop}[theorem]{\bf Proposition}

\newtheorem{question}[theorem]{\bf Question}
\newtheorem{remark}[theorem]{\bf Remark}
\newtheorem{defi}[theorem]{\bf Definition}

\newcommand\cG{{\mathcal G}}

\newcommand\cP{{\mathcal P}}

\newcommand{\Z}{\mathbb{Z}}

\newcommand{\E}{\mathbb{E}}

\title{Avoiding  configurations of small size in the square grid}
\author{
	Máté  Jánosik \\
	Eötvös Loránd University,  \\
	\texttt{janosikmate6@gmail.com}
	\and
	Artúr Nádor \\
	Eötvös Loránd University,  \\
	\texttt{nador.artur@gmail.com} 
	\and
	Zoltán Lóránt    Nagy\thanks{The author is supported by the J\'anos Bolyai Scholarship of the Hungarian Academy of Sciences and by the NRDI EXCELLENCE-24 grant no. 151504 Combinatorics and Geometry.} \\
	Eötvös Loránd University,\\ 
	\texttt{zoltan.lorant.nagy@ttk.elte.hu}
	\and
	László Bence   Simon\\
	University of Cambridge, \\
	\texttt{lacisimon2005@gmail.com} 
}
\date{}

\begin{document}
	
	\maketitle
	\begin{abstract}
		We study the maximum size of a subset of the $n \times n$ integer grid that does not contain specific geometric configurations, a variation of the classical problems initiated by Erdős and Purdy. While extremal problems for 3-point patterns, such as collinear triples and right triangles, are well-studied, the landscape for 4-point configurations in the grid remains less explored. In this paper, we survey the state-of-the-art regarding forbidden 3-point and 4-point configurations, including parallelograms, trapezoids, and concyclic sets. Furthermore, we prove new lower bounds for grid subsets avoiding rhombuses and kites. Specifically, by combining the probabilistic method with the arithmetic properties of Sidon sets, we show that the maximum size of a rhombus-free subset is $\Omega(n^{4/3}(\log n)^{-1/3})$. We also provide near-quadratic lower bounds for sets avoiding kites with axis-parallel diagonals using Behrend-type constructions and discuss implications for square-free sets. These results illustrate the strong interplay between discrete geometry and additive combinatorics.
	\end{abstract}
	\section{Introduction}
	
	Erdős and Purdy  initiated the systematic study of   the 
	distribution of certain $k$-point configurations in finite point sets in the Euclidean plane \cite{EP1, EP2}, and  their work lead to an influential branch of combinatorial geometry with several  interactions with additive combinatorics, extremal set theory, random graph theory, incidence geometry and polynomial methods.
	
	In several cases, the maximum number of configurations is proved or conjectured to be in subsets of an integer grid.
	Thus, while early research focused primarily on the Euclidean plane $\mathbb{R}^2$, recent decades have seen a growing interest towards extremal problems defined on the integer grid $[n]^2 \subset \mathbb{Z}^2$, see  the excellent books of Brass, Moser and Pach \cite{research}, Matousek \cite{Mat} and of Eppstein \cite{Epp}, or similar-flavour extremal problems over vector spaces  see \cite{Ben, Io}.
	This discrete setting offers a rich interplay between geometry, additive combinatorics, and number theory. Furthermore, the rapid development of the polynomial method has provided powerful algebraic tools to address these incidence problems, often leveraging results from finite geometries \cite{GuthKatz15, Tao}.
	In this paper, 
	our central question is as follows:
	
	\begin{question}\label{main_q}
		\emph{What is the maximum number of points one can select from the $n\times n$ square grid that avoids forming certain forbidden configurations?}
	\end{question}

	While the case of $k$-point configurations with $k=2$ is closely related to the distinct distances problem, see \cite{Erdos_dist, dist_vectors, Lef, Sheffer}, the case $k=3$ (e.g., certain triangles, collinear triples) is also well-studied, the landscape for $k=4$ (quadrilaterals) remains less explored.
	Such problems have recently gained increasing attention, for instance, forbidding several types of $4$-point configurations \textit{at the same time} was highly relevant in the solution to one of Erdős' favorite problems on distinct distances by Tao \cite{Tao}, see also \cite{Dumi}.
	
	In this paper we survey some of the most studied variants of Question \ref{main_q} focusing on the cases where the forbidden configuration is of size $k=3$ or $k=4$, and prove  new results in the case of quadrilaterals,  highlighting how the symmetries of the forbidden shape and the arithmetic properties of the grid influence the extremal number. Specifically, by combining the probabilistic method with the arithmetic properties of Sidon sets and $3$-AP free sets of integers, we show lower bounds on the maximum size of rhombus-free subsets. We also provide near-quadratic lower bounds for sets avoiding kites with axis-parallel diagonals  and corollaries  for square-free sets. Finally, we prove an exact result on grid subsets avoiding non-degenerate parallelograms.\\

	\section{Three-point configurations}
	
	In this section, we collect some of the most significant results concerning $3$-point configurations.
	Let $[n]^2 := \{1, 2, \dots, n\} \times \{1, 2, \dots, n\}$ denote the $n \times n$ integer grid in the Euclidean plane. 
	
	One of the oldest and 
	most extensively studied  questions in view   is the so-called \textit{no-three-in-line problem} which was raised by Dudeney \cite{Dudeney}. We formulate its generalisation, which was introduced by Lefmann \cite{Lefmann}.
	
	\begin{defi}
		Let $f_{(k+1)-coll}(n)$ denote the maximum size of a point set chosen from the points  of  the $n\times n$ grid $\cG=[n]\times [n]$ in which at most $k$ points lie on any line.
	\end{defi}
	
	Using this notation, a 50-year-old
	cornerstone result is  due to Hall, Jackson, Sudbery, and Wild.
	
	\begin{theorem}[Hall, Jackson, Sudbery, and Wild \cite{hall1975some}]\label{hall}
		\[ 1.5n-o(n) \le f_{3-coll}(n)\le 2n.\]
	\end{theorem}
	
	It is still open to decide even the asymptotic behavior of the function above.

	We now consider the pattern of isosceles triangles.
	The problem of determining  the maximum number $f_{iso\triangle}(n)$ of  points  we can choose in the $n\times  n$ square grid, so that no three of them form the
	vertices of a (possibly flat) isosceles triangle was posed independently by Wu \cite{Wu},  
	Ellenberg–Jain \cite{Ellen}, and it is also attributed to Erdős \cite{AI}.
	
	To our best knowledge, the best known lower bound on $f_{iso\triangle}(n)$ is sublinear due to Charton, Ellenberg, Wagner, and  Williamson \cite{AI}. They remarked that most probably a linear lower bound can be achieved via the random independent set process, and conjecture that the answer is indeed linear, based on the constructions found by Patternboost \cite{AI}. The best known upper bound is almost quadratic,  $f_{iso\triangle}(n)\le \exp(-c\cdot\log^{1/9}n)n^2$ based on the bound on $3$-term AP-free sets \cite{AI, KM}.
	
	A special case of isosceles triangles are isosceles right triangles. A corner is an axis-parallel isosceles right triangle, i.e., one with legs  parallel to the axes.
	Ajtai and Szemerédi proved in \cite{ASz} that for
	every axis-parallel corner-free subset of $[n] \times [n]$  has size $o(n^2)$.  The best bound is due to Shkredov.
	
	\begin{theorem}[Shkredov, \cite{Sh}]
		For sufficiently large $n$, every subset of $[n] \times [n]$ of size
		at least $n^2/(\log \log n)^C$
		contains corners, three points with coordinates
		$\{(a, b),(a + d, b),(a, b + d)\}$, where where $C> 0$ is an absolute constant. 
	\end{theorem}
	
	Concerning the case of  \textit{isosceles right triangles in general position}, Prendiville showed similar upper bound $n^2/(\log \log n)^C$ \cite{Pren}, which was subsequently improved by Bloom \cite{Bloom} and then Pilatte \cite{Pilatte}, see also \cite{SS}.
	
	\begin{theorem}[Pilatte \cite{Pilatte}] The maximum size isosceles right triangle free set in $[n]^2$ is of size  $O(n^2/(\log n)^{1+\varepsilon})$ for every $\varepsilon>0$.
	\end{theorem}

	\section{Four-point configurations}

	\subsection{Collinear $4$-tuples}
	
	The natural generalisation of the no-$3$-in-line problem is to forbid collinear $4$-tuples in the $n\times n$ grid. This has been studied by  Lefmann \cite{Lefmann}. A very recent improvement on $f_{4-coll}(n)$ is due to Kovács, Nagy, and Szabó \cite{knsz}. For results on the general  no-$(k+1)$-in-line problem, see also \cite{kwan, kovacs2025settling}. 
	
	\begin{theorem}[Kovács-Nagy-Szabó \cite{knsz}]\label{k=3}     $  f_{4-coll}(n)\le 3n$ holds for all $n$ and
		$ 1.973 n\le f_{4-coll}(n)$ for $n$ sufficiently large.
	\end{theorem}
	
	\subsection{Parallelograms}
	
	A \emph{parallelogram} in $[n]^2$ is any quadruple of points
	\[
	(x_1,y_1),\ (x_2,y_2),\ (x_3,y_3),\ (x_4,y_4)
	\]
	satisfying
	\begin{equation}\label{eq:paral}
		(x_1,y_1)+(x_3,y_3) = (x_2,y_2)+(x_4,y_4).
	\end{equation}
	We call a set $\cP\subset [n]^2$ \emph{parallelogram-free} if such an equality does
	not hold except in the trivial case where the four points are not all distinct.
	
	\begin{defi}
		$f_{para}(n)$ denotes the maximum number of points that can be chosen from the grid $[n]^2$ without forming a parallelogram.
	\end{defi}

	Equivalently, $\cP$ is parallelogram-free if all pairwise sums of its points are
	distinct, that is,
	\[
	(a_1+a_2)=(a_3+a_4)\quad\Longrightarrow\quad \{a_1,a_2\}=\{a_3,a_4\}
	\] holds for the elements $a_i\in [1,n]^2$ of the set.
	Thus a parallelogram-free subset of the grid is exactly a
	\emph{two-dimensional Sidon set} (or $B_2$-set) in the Abelian group $\mathbb Z^2$. This enables us to recall the classical result of Lindström, and formalize the result for    $f_{para}(n)$.

	\begin{theorem} [Lindström \cite{Lindstrom}]
		$f_{para}(n) = n + O(n^\frac{2}{3})$
	\end{theorem}
	
	It is natural to ask how the maximum would change if we only forbid non-degenerate parallelograms, i.e., collinear solutions of \eqref{eq:paral} are permitted.

	\begin{defi}
		$f_{nd\_para}(n)$ denote the maximum number of points one can choose from the grid $[1,n]^2$ without forming a non-degenerate parallelogram. 
	\end{defi}
	
	Our result is as follows.
	
	\begin{prop}\label{nondegpar}
		$f_{nd\_para}(n)=2n-1$.
	\end{prop}
	\begin{proof} The union of the point set of a horizontal and a vertical line $\{(x,y_0): x\in [1,n]\}\cup \{(1,y): y\in [1,n]\}$ with $ y_0 \in [1,n]$ provides a possible point set of size $2n-1$ with the required property.
		
		To prove the upper bound $f_{nd\_para}(n)\le 2n-1$, take a point set $\cP$  which avoids non-degenerate parallelograms and  consider the set of distances of its point pairs having the same $y$-coordinate.  The distances can only be the integers in $[1, n-1]$, so $n-1$ different values can be admitted. Each value can only occur in at most one row, otherwise we would have a parallelogram with a horizontal base. Let $k_i$ denote $k_i:=|\ell[y=i]\cap \cP|$. Then for every horizontal line $\ell[y=i]$ which intersect $\cP$ yields at least $k_i-1$ distinct distances.   As a consequence, $$|\cP|=\sum_{i=1}^n k_i, \text{ \ \  while \ \ }\sum_{i=1}^n \max\{(k_i-1), 0\}\le n-1.$$
		Hence, $|\cP|\le 2n-1$.
	\end{proof}

	\subsection{Isosceles trapezoids, concylic four-tuples}
	
	An analogue of the no-$3$-in-line problem is also attributed to  Erdős and Purdy, where the forbidden configuration is a concyclic  set of four points. Here collinear sets are also considered concyclic as a general case. Let $f_{cyclic}(n)$ denote the maximum size of point set of the grid $[n]^2$ which  avoids a concyclic $4$-tuple. Then the upper bound of Thiele \cite{Thiele2} and the recent lower bound of Dong and Xu \cite{Dong} shows the result below.
	
	\begin{theorem}[\cite{Thiele2, Dong}]
		$n-o(n)\le f_{cyclic}(n)\le 2.5n-1.5.$
	\end{theorem}
	
	Note that the approach of Dong and Xu is algebraic geometric and improved upon a  previous algebraic construction of Thiele \cite{Thiele}, which had size $n/4$. We  mention that an even more recent randomized construction by Ghosal, Goenka and Keevash \cite{Keevash} also improved significantly the  lower bound of Thiele, providing $\frac{7}{12}n\le f_{cyclic}(n)$.
	
	While the lower and upper bounds are relatively close,  the upper bound of Thiele actually can be derived by forbidding a far less general type of configuration.  Let $f_{ax-isotr}(n)$ denote the maximum size  of point set of the grid $[n]^2$ which  avoids isosceles trapezoids having an axis-parallel side pair.  Then
	\begin{equation}
		f_{cyclic}(n)\le f_{ax-isotr}(n)\le 2.5n-1.5 
	\end{equation}holds as well \cite{Thiele2}.

	\subsection{Rhombuses}

	Let 
	$f_{rho}(n)$ denote the maximum size of a point set $ \cP \subseteq [1,n] \times [1,n] $ that does not contain the point set of a rhombus.

	\begin{theorem}\label{rhombus}
		$f_{rho}(n)= \Omega\left(\frac{n^{4/3}}{(\log n)^{1/3}}\right)$.
	\end{theorem}

	
	In order to prove Theorem \ref{rhombus}, we will need some lemmata.
	
	\begin{lemma}[Benito, Varona, \cite{pythagorean}]\label{pitlemma}
		The number of ordered Pythagorean triples $(a,b,c)$ such that $a<n$, $b<n$ and $a^2+b^2=c^2$
		is $O(n\log n)$.
	\end{lemma}
	
	\begin{prop}\label{rhombnm} Let $S$ be a Sidon set in $[1,n]$. Then the number of   rhombuses in $[1,n]\times S$ is $O(n^2\log n)$.
	\end{prop}
	
	\begin{proof} First we give an upper bound on the rhombuses on the point set $[1,n]\times S$ which are not axis-parallel squares.  Suppose that the vertices of a rhombus $ABCD$ are points of $[1,n]\times S$. 
		The coordinates of the vertices $(a_1,a_2),(b_1,b_2),(c_1,c_2),(d_1,d_2)$ of a rhombus  satisfy $a_2+c_2=b_2+d_2$. As these four coordinates are chosen from a Sidon set, we have $\{a_2,c_2\}= \{b_2,d_2\}$. This means that two bases of the rhombus are horizontal. 
		
		Without loss of generality, suppose that $a_1$ is the smallest first coordinate of the four and $b_1$ is the smallest first coordinate on the other horizontal base. This implies $b_2=c_2$ and $a_2=d_2$. Since the rhombus is not an axis-parallel square, the three points $(a_1,a_2),(a_1,b_2),(b_1,b_2)$ form a non-degenerate right triangle, with legs of positive length $a:=b_1-a_1<n$ and $b:=|a_2-b_2|<n$. 
		The point $A=(a_1,a_2)$ and the pair $(a,b)$ determine the rhombus. Every integer $b$ occurs at most once as a difference of two numbers from the Sidon set. 
		There are $O(n\log n)$ pairs $(a,b)$ such that $a<n$, $b<n$ and $a^2+b^2$ is a square according to  Lemma~\ref{pitlemma}.
		Once $b$ is fixed, there are at most $2n$ possible  positions for $A$.
		Thus, in total, the number of rhombuses in $[1,n]\times S$ which are not axis-parallel squares is $O(n^2\log n)$. \\
		Finally, observe that the number of axis-parallel squares is $O(n^2)$, as there are at most $n$ segments that can be obtained as the projection of the square to the $y$ axis, taking into account that $S$ is  Sidon set. This completes the proof.
	\end{proof}
	
	\begin{proof}[Proof of Theorem \ref{rhombus}] Let $S\subset[1,n]$ be a Sidon set of maximum size. Hence, $|S|=(1+o(1))\sqrt{n}$.
		We use a probabilistic method with a deletion argument to show the existence of a rhombus-free set of $[1,n]^2$ of the required size. 
		
		Let $q \in [0,1]$ be a fixed probability, which will be determined later. We choose a random point set $Q$ from $[1,n]\times S$ by independently picking every point with  probability $q$. We create a set $Q'$ by leaving one point from each rhombus in $Q$. There are $O(n^2\log n)$ rhombuses in $[1,n]\times S$ in view of Proposition~\ref{rhombnm}.
		Taking expected values on the number of points in $Q'$, we obtain
		$$\E(|Q'|)\ge \E(|Q|)-\E(|\{\text{rhombuses in $Q$}\}|)$$
		$$\E(|Q'|)\ge q(1+o(1)){n^{3/2}}-q^4Cn^2\log n,$$ with a suitable positive constant $C$.
		Setting $q=\frac{n^{-1/6}}{2(C\log n)^{1/3}}$, we get
		$$\E(|Q'|)\ge \frac38\frac{n^{\frac43}}{(C\log n)^{\frac13}},$$ which shows the existence of a point set $\cP\subseteq \cG$ of size $\displaystyle f_{rho}(n)= \Omega\left(\frac{n^{4/3}}{(\log n)^{1/3}}\right)$ without rhombuses.
	\end{proof}
	\begin{remark} To put this lower bound into perspective, we remark that the function $f_{rho}(n)$ is closely related to problem of maximum size point sets without isosceles triangles. Indeed, it is easy to see that every triple of the $4$ vertices of a rhombus form an isosceles triangle while the vertex set of  every isosceles triangle can be completed to the vertex set of a rhombus in $\Z\times \Z$. In fact, a bound of $\Omega(n^{1.242})$ for $f_{rho}(n)$ can be deduced via the alteration method together with the state-of-art extremal result of Pach and Tardos, claiming that every $N$-point planar set spans at most 
		$O(N^{2.136})$ isosceles triples \cite{PT}. 
	\end{remark}
	\subsubsection{Generalization using difference sets $B_2^-[g]$}
	
	It is natural to ask whether the lower bound for $f_{rho}(n)$ can be improved by selecting the $y$-coordinates from a denser set than a Sidon set. We investigate this by considering generalized Sidon sets.
	While the standard definition of $B_2[g]$ sets restricts the number of solutions to $a+b=m$, for our geometric purposes regarding rhombuses, we require a constraint on the differences.
	
	\begin{defi}
		A set $S \subseteq \mathbb{Z}$ is called a $B_2^-[g]$ set (or a set with the property that differences are repeated at most $g$ times) if for every integer $d \neq 0$, the equation 
		\[ x - y = d, \quad x, y \in S \]
		has at most $g$ solutions.
	\end{defi}
	
	Note that for $g=1$, this is equivalent to the classical Sidon set definition. For $g > 1$, this class differs from the standard sum-based $B_2[g]$ sets, but their asymptotic behavior regarding maximum density is analogous. 
	It is a known result (Schmutz - Tait \cite{Tait} and Xu \cite{Xu}, see also Cilleruelo, Ruzsa, and Trujillo \cite{Cilleruelo02}) that the maximum size of a $B_2^-[g]$ set in $[1, n]$ is $|S_g| = \Theta(\sqrt{gn})$.
	We construct our point set $\mathcal{P}$ by taking a random subset of the grid $ [1, n] \times S_g$. The size of this base grid is $|[1, n] \times S_g| = \Theta(\sqrt{n^{3}g})$.
	
	We now estimate the number of rhombuses in $[1, n] \times S_g$. A rhombus is determined by its side lengths and the coordinates of its vertices. For a rhombus to exist on the grid, its side lengths must satisfy geometric constraints related to Pythagorean triples.
	
	In our product set construction, a rhombus is formed by choosing:
	\begin{enumerate}
		\item A vertical difference $h$ (height).
		\item Two pairs of $y$-coordinates $\{y_1, y_2\}$ and $\{y_3, y_4\}$ from $S_g$ such that $y_2 - y_1 = y_4 - y_3 = h$. Since $S_g$ is a $B_2^-[g]$ set, there are at most $g$ pairs with difference $h$. Thus, there are at most $\binom{g}{2} = O(g^2)$ ways to choose the vertical positions of the left and right sides of the rhombus.
		\item A valid horizontal shift $s$ and length $w$ such that $s^2 + h^2 = w^2$.
		\item An $x$-coordinate for the bottom-left vertex (approx $n$ choices).
	\end{enumerate}
	
	Summing over all possible heights and valid triples via Lemma \ref{pitlemma}, the number of rhombuses $N_{rho}$ in $[1, n] \times S_g$ is bounded by:
	\begin{equation}
		N_{rho} \le C \cdot n \cdot g^2 \cdot (n \log n) = O(n^2 g^2 \log n).
	\end{equation}
	
	We apply the alteration method. Let $\mathcal{P}'$ be a random subset of $[1, n] \times S_g$ where each point is included with probability $p$. We then remove one point from each rhombus. The expected size of the final set is:
	\begin{equation}
		\mathbb{E}(|\mathcal{P}_{final}|) = p |\mathcal{G}_g| - p^4 N_{rho} = c_1 p n^{3/2}g^{1/2} - c_2 p^4 n^2 g^2 \log n.
	\end{equation}
	To maximize this, we choose $p$ such that the two terms are of the same order:
	\[
	p \approx \left( \frac{n^{3/2}g^{1/2}}{n^2 g^2 \log n} \right)^{1/3} = n^{-1/6} g^{-1/2} (\log n)^{-1/3}.
	\]
	Substituting this $p$ back into the expectation:
	\begin{equation}
		\mathbb{E}(|\mathcal{P}_{final}|) \approx n^{-1/6} g^{-1/2} (\log n)^{-1/3} \cdot n^{3/2} g^{1/2} = \Theta\left( \frac{n^{4/3}}{(\log n)^{1/3}} \right).
	\end{equation}
	\begin{remark}
		The parameter $g$ cancels out completely in the final bound. This implies that increasing the density of the rows by relaxing the Sidon condition to $B_2^-[g]$ is exactly counterbalanced by the quadratic increase in the number of rhombuses. The resulting lower bound $\Omega(n^{4/3}(\log n)^{-1/3})$ is robust and matches the result derived from standard Sidon sets.
	\end{remark}

	\subsection{Kites}

	\begin{defi}
		Let us denote by $f_{ax\_kite}(n)$ the maximum size of a point set $\cP\subseteq [1,n]\times [1,n]$, with the condition that $\cP$ does not contain kites having horizontal and vertical diagonals. 
	\end{defi}
	
	\begin{prop}\label{apk1}
		$f_{ax\_kite}(n)=\Omega\left(\frac{n^2}{e^{O(\sqrt{\log{n}})}}\right)$.
	\end{prop}
	
	In order to prove the bound, we recall the celebrated result of Behrend.
	
	\begin{lemma}[Behrend \cite{Behrend}]\label{Be}
		There exists a 3-term arithmetic progression (3-AP, for brief) free set $S \subseteq \{1, 2, \dots, n\}$ of size $|S|=\Omega\left(\frac{n}{e^{O(\sqrt{\log{n}})}}\right)$.
	\end{lemma}
	
	\begin{proof}[Proof of Proposition \ref{apk1}] Take a $3$-AP free set of maximum size in $[1,n]$. Consider the set $S\times S\subseteq [1,n]^2$. Kites with axis-parallel diagonals have a diagonal that is an axis of symmetry. A projection in the direction of this diagonal would give rise to a $3$-AP, but vertical or horizontal projections of $S\times S$ does not contain such configurations, by definition.
	\end{proof}

	\begin{prop}\label{apk2}
		$f_{ax\_kite}(n)=o(n^2)$
	\end{prop}
	
	\begin{proof}
		Partition the point set $[1,n]^2$ into two subsets $T_1$ and $T_2$ according to the parity of the sum of their coordinates. That is, a point $(x, y) \in \cP$ belongs to $T_1$ if $x + y$ is even, and to $T_2$ otherwise. Each $T_i$ is part of a square grid in which the unit length is $\sqrt{2}$, moreover, both $T_1$ and $T_2$  are contained in distinct $n\times n$ square grids. Suppose to the contrary that it is possible to select $cn^2$ points from $[1,n]^2$ such that no four of them form an axis-parallel kite, where $c>0$ is a constant.
		
		Then, by the pigeonhole principle, at least $\frac{c}{2}n^2$ of these points lie in either $T_1$ or $T_2$, thus lie in an $n\times n$ square grid. By the assumption, these sets do not contain kites with axis-parallel diagonals.  However,in this case both has to be of order of magnitude $o(n^2)$ due to the Fürstenberg-Katznelson result (Theorem \ref{Solym}) as they can not contain axis-aligned squares either.
	\end{proof}
	
	\subsection{Rectangles, Squares}
	
	In this section we recall some classical results with respect to rectangles and squares.
	
	\begin{defi}
		Let us denote by $f_{ax\_rec}(n)$ the maximum size of a point set $\cP\subseteq [1,n]\times [1,n]$, with the condition that $\cP$ does not contain axis-parallel rectangles. \\
		Let us denote by $f_{ax\_sq}(n)$ the maximum size of a point set $\cP\subseteq [1,n]\times [1,n]$, with the condition that $\cP$ does not contain axis-parallel squares. Likewise, let us denote by $f_{sq}(n)$ the maximum size of a point set $\cP\subseteq [1,n]\times [1,n]$, with the condition that $\cP$ does not contain squares.
	\end{defi}

	The corresponding problem w.r.t. rectangles is well-known, as it was asked in a slightly modified formulation by Zarankiewicz \cite{Zara}. Indeed, this problem is equivalent to determining the maximum number of edges in a balanced bipartite graph with $n$ vertices on each side, under the condition that the graph contains no $4$-cycles (i.e., it is $C_4$-free). 
	
	\begin{theorem}[Reiman \cite{Reiman}] {$$(1+o(1))n^{\frac{3}{2}}\le f_{ax\_rec}(n)\le \frac{n}{2}\bigl(1+\sqrt{4n-3}\bigr).$$
			The upper bound is sharp when there is a $q$, such that $n=q^2+q+1$ and there exists a finite projective plane of order $q$.
		}
	\end{theorem}
	\noindent For further exact results on Zarankiewicz numbers $Z(n,n,2,2)$, see \cite{Furedi, FS, KST}.\\
	
	Now we consider the case when the forbidden configurations are squares.
	
	\begin{theorem}[Fürstenberg-Katznelson \cite{FK}, see also Solymosi's work \cite{Solymosi}]\label{Solym}
		
		$$f_{ax\_sq}(n)=o(n^2).$$
	\end{theorem}
	\noindent Fürstenberg and Katznelson actually proved a much more
	general theorem via ergodic theory, while Solymosi's proof is combinatorial. The problem was raised by Erdős and Graham \cite{EG}, see also \cite{ASz}.

	\begin{prop}\label{axsq}
		$$f_{ax\_sq}(n)=\Omega\left(\frac{n^2}{e^{O(\sqrt{\log{n}})}}\right).$$    
	\end{prop}

	\begin{proof} Take the  grid points of $[n]^2$ in which the sum of the coordinates is even. The resulting  set of points contains 
		a $\lfloor{\frac n2\rfloor}\times \lfloor{\frac n2\rfloor}$ square grid $\cG$, which has axis-parallel diagonals w.r.t. the original grid $[n]^2$.  
		We apply the construction presented in Proposition \ref{apk1}, and define a $3$-AP-free set $S$ within $\lfloor{\frac n2\rfloor}$ to pick the points of $S\times S$ in $\cG$.  $S\times S$ is an
		axis-parallel kite-free construction hence it gives us an axis-parallel square-free construction in $\cG$, thus it does not contain  squares having diagonals parallel to the sides of $\cG$. Crucially, this means that it does not contain  squares having horizontal and vertical sides of the original grid $[n]^2$.
		The size of the construction $S\times S$ implies the claimed bound.
	\end{proof}

	Using Singer's construction for maximal Sidon sets, we present a lower bound for general square-free sets as well.
	
	\begin{lemma}[Singer  \cite{Singer-Sidon}, see also \cite{ Erdos-Sidon}]\label{Singersidon} There exists a perfect difference set in $\Z_m$ for $m=q^2+q+1$ if $q$ is a prime power. Its cyclic translates are also perfect difference sets in the group. As a consequence,
		there exist a Sidon sets $S \subseteq \{1, 2, \dots, n\}$ of size $|S|=\sqrt{n}(1-o(1))$.
	\end{lemma}

	\begin{theorem}\label{sq}
		$$f_{sq}(n)=\Omega\left(\frac{n^{3/2}}{e^{O(\sqrt{\log{n}})}}\right).$$
	\end{theorem}

	\begin{proof}
		We construct a point set that is both axis-parallel and non-axis-parallel square-free. Consider a point set $\cP_1\subset [m]^2$ of maximal size which does not contain axis-parallel squares, where $m$ is of form $q^2+q+1$ for some prime power $q$. Such a set was constructed in the proof of Proposition \ref{axsq} with size $|\cP_1|=\Omega\left(\frac{n^2}{e^{O(\sqrt{\log{n}})}}\right).$ Now take a
		perfect difference set in a cyclic group $\Z_m$ for $m=q^2+q+1$, where $\frac{n}{2}<m\le n$. It can be embedded to $[1,m]\subseteq [1,n]$ so that the resulting set is a Sidon set. The same holds for its translates. By averaging, for one of the translates $S$ of the perfect difference set, $S\times [1,n]$  intersects $\cP_1$ in at least $\frac{q+1}{q^2+q+1}|\cP_1|$ points. The resulting point set $\cP_2:=S\cap \cP_1$ has size $\Omega\left(\frac{n^{3/2}}{e^{O(\sqrt{\log{n}})}}\right)$. It also does not contain squares, because not axis-aligned squares would have vertices having distinct $x$-coordinates $x_1, x_2, x_3, x_4$ such that  $x_2-x_1=x_3-x_4$ or vertices having 3 distinct $x$-coordinates $x_1, x_2, x_3, x_4$ such that  $x_2=x_4=(x_1+x_3)/2$. In either case,  we would get a contradiction to the Sidon property of $S$.
	\end{proof}

	We conjecture that this lower bound might be improved significantly. As for the upper bound, we know that the order of magnitude is subquadratic.
	
	\begin{theorem}[Prendiville \cite{Pren}, see also Shkredov-Solymosi, \cite{SS}]\label{sq_up}
		$$f_{sq}(n)=O\left(\frac{n^{2}}{(\log\log{n})^c}\right),$$ where $c>0$ is an absolute constant.
	\end{theorem}


\footnotesize
	\begin{thebibliography}{pippo}
		
		\bibitem{ASz} Ajtai, M. and Szemerédi, E. (1974). Sets of lattice points that form no squares.  Studia Scientiarium
		Mathematicarum Hungarica 9, 9–11.
		
		\bibitem{Behrend} Behrend, F. A. (1946). On sets of integers which contain no three terms in arithmetical progression. \textit{Proceedings of the National Academy of Sciences of the United States of America, 32}(12), 331–332. 
		
		\bibitem{pythagorean} Benito, M., \& Varona, J. L. (2002). Pythagorean triangles with legs less than n. \textit{Journal of Computational and Applied Mathematics, 143}(1), 117–126. 
		
		\bibitem{Ben} Bennett, M., Iosevich, A. \& Pakianathan, J. Three-point configurations determined by subsets of 
		via the Elekes-Sharir Paradigm. Combinatorica 34, 689–706 (2014).
		
		\bibitem{Bloom} Bloom, T. F. (2014). Quantitative results in arithmetic combinatorics (Doctoral dissertation, University of Bristol).
		
		\bibitem{research} Brass, P., Moser, W. O., \& Pach, J. (2005). \textit{Research problems in discrete geometry}. Springer.
		
		\bibitem{Cilleruelo02}
		Cilleruelo, J.,  Ruzsa, I. Z.,  and  Trujillo, J.
		\textit{Upper and lower bounds for finite $B_h[g]$ sequences},
		Journal of Number Theory, \textbf{97}(1) (2002), 26--34.
		
		
		\bibitem{AI} Charton, F., Ellenberg, J. S., Wagner, A. Z., \& Williamson, G. (2024). Patternboost: Constructions in mathematics with a little help from 
		AI. arXiv preprint arXiv:2411.00566. 
		
		\bibitem{Dong} Dong, Z., \& Xu, Z. (2025). Large grid subsets without many cospherical points. arXiv preprint arXiv:2506.18113.
		
		\bibitem{Dudeney} 
		Dudeney, H. E. (1917). {\em Amusements in mathematics} (Vol. 473). Courier Corporation.
		
		\bibitem{Dumi} Dumitrescu, A. (2020). Distinct distances in planar point sets with forbidden 4-point patterns. Discrete Mathematics, 343(9), 111967.
		
		\bibitem{Ellen} Ellenberg, J. S., \& Jain, L. (2019). Convergence rates for ordinal embedding. arXiv preprint arXiv:1904.12994.
		
		\bibitem{Epp} Eppstein, D. (2018). \textit{Forbidden configurations in discrete geometry}. Cambridge University Press.
		
		\bibitem{Erdos-Sidon} Erdős, P. (1944). On a problem of Sidon in additive number theory and on some related problems addendum. \textit{Journal of the London Mathematical Society, 1}(s1-19), 208.
		
		\bibitem{Erdos_dist} Erdős, P., \& Guy, R. K. (1970). Distinct distances between lattice points. \textit{Elemente der Mathematik, 25}, 121–123.
		
		\bibitem{EG} Erdős, P., (1973) Problems and results on combinatorial number theory. In A Survey of
		Combinatorial Theory (Proc. Internat. Sympos., Colorado State Univ., Fort Collins, Colo., 1971),
		North-Holland, Amsterdam, pp. 117–138.
		
		\bibitem{dist_vectors} Erdős, P., Graham, R. L., Ruzsa, I. Z., \& Taylor, H. (1992). Bounds for arrays of dots with distinct slopes or lengths. \textit{Combinatorica, 12}(1), 39–44. 
		
		\bibitem{EP1} Erdős, P., \& Purdy, G. (1971). Some extremal problems in geometry. \textit{Journal of Combinatorial Theory, Series A, 10}(3), 246–252. 
		
		\bibitem{EP2} Erdős, P., \& Purdy, G. (1976). Some extremal problems in geometry IV. \textit{Congressus Numerantium, 17}, 307–322.
		
		\bibitem{Furedi} Füredi, Z. (1996). New asymptotics for bipartite Turán numbers. Journal of Combinatorial Theory, Series A, 75(1), 141–144.
		
		\bibitem{FS} Füredi, Z., and Simonovits, M. (2013). The history of degenerate (bipartite) extremal graph problems. In Erdős centennial (pp. 169-264). Berlin, Heidelberg: Springer Berlin Heidelberg.
		
		\bibitem{FK} Furstenberg, H., \& Katznelson, Y. (1991). A density version of the Hales-Jewett theorem. Journal d’Analyse Mathématique, 57(1), 64-119.
		
		\bibitem{Keevash} Ghosal, A., Goenka, R., \& Keevash, P. (2025). On subsets of lattice cubes avoiding affine and spherical degeneracies. ArXiv preprint arXiv:2509.06935.
		
		
		\bibitem{kwan} Grebennikov, A., \& Kwan, M. (2025). No-$(k+ 1) $-in-line problem for large constant $ k$. arXiv preprint arXiv:2510.17743.
		
		
		
		\bibitem{GuthKatz15}
		Guth, L. \&  Katz, N. H.
		\textit{On the Erd\H{o}s distinct distances problem in the plane}, 
		Annals of Mathematics, \textbf{181}(1) (2015), 155--190.
		
		
		
		\bibitem{hall1975some}
		Hall, R. R., Jackson, T. H., Sudbery, A.,    Wild, K. (1975). Some advances in the no-three-in-line problem. {\em Journal of Combinatorial Theory, Series A}, {\em 18}(3), 336--341.
		
		\bibitem{Io} Iosevich, A., Koh, D., \& Pham, T. (2020). New bounds for distance-type problems over prime fields. European Journal of Combinatorics, 86, 103080.
		
		\bibitem{KM} Kelley, Z., \& Meka, R. (2023, November). Strong bounds for 3-progressions. In 2023 IEEE 64th Annual Symposium on Foundations of Computer Science (FOCS) (pp. 933-973). 
		
		\bibitem{kovacs2025settling}
		Kov{á}cs, B., Nagy, Z. L.,    Szab{ó}, D. R. (2025). Settling the no-(k+1)-in-line problem when k is not small. {\em arXiv preprint arXiv:2502.00176}.
		
		
		\bibitem{knsz} Kovács, B., Nagy, Z. L.,  Szabó, D. R. (2025). Randomised algebraic constructions for the no-$(k+ 1) $-in-line problem. arXiv preprint arXiv:2508.07632.
		
		\bibitem{KST}
		Kővári, T., Sós, V. T., \& Turán, P. (1954). On a problem of K. Zarankiewicz. Colloquium Mathematicum, 3(1), 50–57.
		
		\bibitem{Lefmann} Lefmann, H. (2012). Extensions of the no-three-in-line problem. preprint.
		
		\bibitem{Lef} Lefmann, H., \& Thiele, T. (1995). Point sets with distinct distances. \textit{Combinatorica, 15}(3), 379–408. 
		
		\bibitem{Lindstrom} Lindström, B. (1972). On $B_2$-sequences of vectors. \textit{Journal of Number Theory, 4}(3), 261–265. 
		
		
		\bibitem{Mat} Matoušek, J. (2013). \textit{Lectures on discrete geometry}. Springer Science \& Business Media.
		
		\bibitem{PT} Pach, J., \& Tardos, G. (2002). Isosceles triangles determined by a planar point set. \textit{Graphs and Combinatorics, 18}(4), 769–779. 
		
		\bibitem{Pilatte} Pilatte, C. (2022). New bound for Roth’s theorem with generalized coefficients. Discrete Analysis 16, 21 pp. 
		
		\bibitem{Pren} Prendiville, S. (2015). Matrix progressions in multidimensional sets of integers. Mathematika, 61(1), 14-48.
		
		\bibitem{Reiman} Reiman, I. (1958). Über ein problem von K. Zarankiewicz. Acta mathematica hungarica, 9(3-4), 269-273.
		
		\bibitem{Tait} Schmutz, E.,  Tait, M. (2025). Cardinalities of $ g $-difference sets. arXiv preprint arXiv:2501.11736.
		
		\bibitem{Sheffer} Sheffer, A. (2014). \textit{Distinct distances: open problems and current bounds}. arXiv preprint arXiv:1406.1949. https://arxiv.org/abs/1406.1949
		
		\bibitem{Sh} Shkredov, I. D. (2009). On a two-dimensional analogue of Szemerédi's theorem in Abelian groups. Izvestiya Rossiiskoi Akademii Nauk. Seriya Matematicheskaya, 73(5), 181-224.
		
		\bibitem{SS} Shkredov, I. D., \& Solymosi, J. (2022). Tilted corners in integer grids. Number Theory and Combinatorics: A Collection in Honor of the Mathematics of Ronald Graham, 329.
		
		
		\bibitem{Singer-Sidon} Singer, J. (1938). A theorem in finite projective geometry and some applications to number theory. Transactions of the American Mathematical Society, 43(3), 377-385.
		
		\bibitem{Solymosi} Solymosi, J. (2004). A note on a question of Erdős and Graham. \textit{Combinatorics, Probability and Computing, 13}(2), 263–267.
		
		
		\bibitem{Tao} Tao, T. (2025). Planar point sets with forbidden 4-point patterns and few distinct distances. Discrete \& Computational Geometry, 1-9.
		
		\bibitem{Thiele} Thiele, T. (1995). The no-four-on-circle problem. Journal of Combinatorial Theory, Series A, 71(2), 332-334.
		
		\bibitem{Thiele2} Thiele, T. (1995). Geometric selection problems and hypergraphs (Doctoral dissertation, Verlag nicht ermittelbar).
		
		\bibitem{Wu}
		Wu, C. W. (2016). Counting the number of isosceles triangles in rectangular regular grids. arXiv preprint arXiv:1605.00180.
		
		\bibitem{Xu} Xu, W. (2018). Popular differences and generalized Sidon sets. Journal of Number Theory, 186, 103-120.
		
		\bibitem{Zara} Zarankiewicz, K. (1951). Problem P 101. Colloquium Mathematicum, 2, 301.
		
	\end{thebibliography}
\end{document}